\documentclass[11pt]{amsart}
\usepackage{graphicx}
\usepackage{amssymb, amsmath}
\newtheorem{theorem}{Theorem}[section]
\newtheorem{corollary}[theorem]{Corollary}
\newtheorem{lemma}[theorem]{Lemma}
\newtheorem{proposition}[theorem]{Proposition}
\theoremstyle{definition}
\newtheorem{definition}[theorem]{Definition}
\newtheorem{example}[theorem]{Example}

\theoremstyle{remark}

\begin{document}
\title{Completions of Lie Algebras}
\author{\sc M. Shahryari}
\thanks{}
\address{ Department of Pure Mathematics,  Faculty of Mathematical
Sciences, University of Tabriz, Tabriz, Iran }

\email{mshahryari@tabrizu.ac.ir}
\date{\today}

\begin{abstract}
A Lie algebra $K$ over a field of characteristic zero $E$ is called a completion of a rational Lie algebra $L$, if it contains $L$ as $\mathbb{Q}$-subalgebra and the $E$-span of $L$ is equal to $K$. The class of all completions of a rational Lie algebra is studied in this article.
\end{abstract}

\maketitle



\vspace{2cm}
Let $L$ be a Lie algebra over $\mathbb{Q}$, the field of rational numbers. A pair $(K, E)$ with $E$ a field characteristic zero and $K$ a Lie algebra over $E$, is a completion of $L$, if $L\leq_{\mathbb{Q}}K$ ( a $\mathbb{Q}$-Lie subalgebra), and $K=Span_E(L)$. In many situations, $\dim_{\mathbb{Q}}L$  is infinite, while $\dim_E K$ is finite. Therefore we can apply some properties of finite dimensional Lie algebras in characteristic zero to $L$. For example, we will show that, if $L$ is nilpotent, then it can be embedded in the algebra of strictly upper triangular $n\times n$ matrices with entries from $E$, for some $n$. In some cases $(L, E)$ is a completion of $L$, for example, let $L=\mathbb{R}^3_{\wedge}$, the infinite dimensional $\mathbb{Q}$-algebra of three dimensional real vectors with the wedge product as the Lie bracket. Although $L$ is infinite dimensional, its completion $(L, \mathbb{R})$ has dimension $3$. This observations motivate us to study the class of completions of rational Lie algebras. Most of the work here can be done for other fields, but we do it just for $\mathbb{Q}$. \\

The type of problems we are going to deal with in this article, are not new; at least in two situations, similar problems are studied. The first one is the study of centroid of Lie algebras. Recall that for a Lie algebra $L$, its centroid is defined by
$$
\Gamma(L)=\{ \alpha\in End(L): \forall x, y\in L\ \ [\alpha(x), y]=\alpha[x,y]\}.
$$
The centroid of $L$ is the largest ring in which $L$ is a Lie algebra over it. So if $(L, E)$ is a completion of $L$, then $E\leq \Gamma(L)$. If $L$ is simple, then its centroid is a field (see \cite{Jac}), so it is the largest field $E$, with $(L, E)$ a completion.\\

The second similar situation can be seen in the theory of exponential groups, where the concept of completion for groups is defined. The concept of centroid can be also defined for groups and it is the largest ring in which the group admits as a ring of scalars. For a complete description of exponential groups see \cite{Rem1} and \cite{Rem2}.\\

This article is arranged in three sections; in the first one, we show that any abelian rational Lie algebra ( a vector space over $\mathbb{Q}$) is potentially one dimensional. In the second section, we introduce the class $\mathfrak{X}_L$ of all completions of $L$ and we show that this class is elementary in the case   $\dim L<\infty$. In the section three, we introduce {\em entangled} ideals of $E\otimes_{\mathbb{Q}}L$ and we study their connections to completions. This  section contains also some results on the completions in some finite dimensional cases.

\section{ Abelian Lie algebras}
In this section, we show by an elementary argument that any abelian Lie algebra ( i.e. any vector space over $\mathbb{Q}$), is potentially one dimensional, i.e. for any vector space $L$ over $\mathbb{Q}$, there exists a field $E$ of characteristic zero, such that $\dim_E L=1$. Then we show that this can not be generalized for nilpotent Lie algebras of class $2$.

\begin{proposition}
Let $L$ be a vector space over $\mathbb{Q}$. Then there exists a field $E$ of characteristic zero, such that $L$ is a vector space over $E$ and $\dim_E L=1$.
\end{proposition}

\begin{proof}
First, let $\dim_{\mathbb{Q}}L=n$ and $x_1, \ldots, x_n$ be a basis of $L$. Suppose
$$
q(t)=a_0+a_1t+\cdots+a_{n-1}t^{n-1}+t^n\in \mathbb{Q}[t]
$$
be an irreducible polynomial. Let $\lambda\in \mathbb{C}$ be any root of $q(t)$ and $E=\mathbb{Q}(\lambda)$. Define an action of $E$ on $L$ by the $\mathbb{Q}$-bilinear extension of following rules
\begin{eqnarray*}
\lambda.x_1&=&x_2\\
\lambda.x_2&=&x_3\\
           &\vdots&\\
\lambda.x_n&=&-a_0x_1-a_1x_2-\cdots-a_{n-1}x_n.
\end{eqnarray*}
It is easy to check that $L$ is a vector space over $E$ of dimension $1$. Now, suppose $\dim_\mathbb{Q}L=\alpha$ is infinite. Let $X$ be any set with the cardinality $\alpha$ and suppose $\mathbb{Q}(X)$ is the field rational functions with variables in the set $X$. Note that
$$
\mathbb{Q}(X)=\bigcup_{Y\subset X}\mathbb{Q}(Y),
$$
where $Y$ varies in the set of all finite subsets of $X$. Since every $\mathbb{Q}(Y)$ is countable so we have $|\mathbb{Q}(X)|=|X|$. Hence as $\mathbb{Q}$-spaces, we have $L\cong\mathbb{Q}(X)$. Let $f:L\to \mathbb{Q}(X)$ be a $\mathbb{Q}$-isomorphism. Now for $q\in \mathbb{Q}(X)$ and $x\in L$, define
$$
q.x=f^{-1}(qf(x)).
$$
So, by assuming $E=\mathbb{Q}(X)$, $L$ is an $E$-space with $\dim_E L=1$.
\end{proof}

One may wants to generalize this proposition for the case of nilpotent Lie algebras of higher nilpotency classes. But this is not true in general, since the three dimensional Heisenberg algebra $H=\langle x, y, z\rangle$ with $[x,y]=z$ and $z\in Z(H)$ is a nilpotent rational Lie algebra of class two but there is no field $E$ with $\dim_E H=2$. In general, the extension degree $[E:\mathbb{Q}]$ is restricted by the numbers $\dim_{\mathbb{Q}}C_L(x)$, $x\in L$. To see this, let $L$ be any Lie algebra over $E$. Then clearly for any $x\in L$, we have $Span_E(x)\subseteq C_L(x)$, so
$$
[E:\mathbb{Q}]=\dim_{\mathbb{Q}}Span_E(x)\leq \dim_{\mathbb{Q}}C_L(x).
$$
Therefore, we have

\begin{lemma}
Let $L$ be a rational Lie algebra and $E$ be any field such that $(L, E)$ is a completion for $L$. Then
$$
[E:\mathbb{Q}]\leq \min_{x\in L}\ \dim_{\mathbb{Q}}C_L(x).
$$
\end{lemma}

By a similar argument, we can prove the following proposition.

\begin{proposition}
Let $L$ be an infinite dimensional rational Lie algebra which has a finite dimensional completion of the form $(L, E)$. Then

\ i- either $L$ is solvable or for any ordinal $\alpha$ we have $\dim_{\mathbb{Q}}L^{(\alpha)}=\infty$.

ii- either $L$ is nilpotent or for any ordinal $\alpha$ we have $\dim_{\mathbb{Q}}\gamma_{\alpha}(L)=\infty$.
\end{proposition}

\section{The class of completions}
In this section, we show that the class of completions of any finite dimensional rational Lie algebra is elementary ( axiomatizable in some first order language). We also show that this is not true for the class of all fields $E$ in which $L$ is a Lie algebra over $E$. Remember that a pair $(K, E)$ is a completion of $L$, if

\ \ i- $E$ is a field of characteristic zero,

\ ii- $K$ is a Lie algebra over $E$,

iii- $L\leq_{\mathbb{Q}}K$ ( a $\mathbb{Q}$-Lie subalgebra),

iv- $K=Span_E(L)$.

Let $\mathfrak{X}_L$ be the class of all completions of $L$.

\begin{theorem}
If $\dim_{\mathbb{Q}}L$ is finite, then the class $\mathfrak{X}_L$ is elementary.
\end{theorem}

\begin{proof}
Let $v_1, \ldots, v_n$ be a $\mathbb{Q}$-basis of $L$. We introduce a first order language
$$
\mathcal{L}=(0, 1, +, \times, q, (a^{\ast}, a\in L))
$$
which contains constant symbols $0$, $1$ and $a^{\ast}$ for all $a\in L$. It also has two binary functional symbols $+$ and $\times$, and $q$ is a unary predicate symbol which will be interpreted as "{\em to be scalar}". Let $\Sigma$ be the following set of axioms:\\

\ 1- $\forall x\forall y\ (q(x)\wedge q(y))\Rightarrow (q(x+y)\wedge q(xy))$.

\ 2- $\forall x\forall y\ (q(x)\wedge q(y))\Rightarrow (x+y=y+x \wedge x\times y=y\times x)$.

\ 3- $\forall x\forall y\forall z\ (q(x)\wedge q(y)\wedge q(z))\Rightarrow ((x+y)+z=x+(y+z)\wedge (x\times y)\times z=x\times (y\times z))$.

\ 4- $\forall x\ q(x)\Rightarrow x+0=x$.

\ 5- $\forall x\ 1\times x=x$.

\ 6- $\forall x\ q(x)\Rightarrow (\exists y\ q(y)\wedge x+y=0)$.

\ 7- $\forall x\forall y\forall z \ (q(x)\wedge q(y)\wedge q(z))\Rightarrow (x\times (y+z)=x\times y+x\times z)$.

\ 8- $\forall x\ (q(x)\wedge x\neq 0)\Rightarrow (\exists y\ q(y)\wedge x\times y=1)$.

\ 9- $\forall x\forall y\ (\neg q(x)\wedge \neg q(y))\Rightarrow(\neg q(x+y)\wedge q(x\times y))$.

10- $\forall x\forall y\ (\neg q(x)\wedge \neg q(y))\Rightarrow (x+y)=y+x$.

11- $\forall x\ \neg q(x)\Rightarrow x\times x=0^{\ast}$.

12- $\forall x\forall y\forall z\ (\neg q(x)\wedge \neg q(y)\wedge \neg q(z))\Rightarrow (x+y)+z=x+(y+z)$.

13- $\forall x\ \neg q(x)\Rightarrow x+0^{\ast}=x$.

14- $\forall x\ \neg q(x)\Rightarrow(\exists y\  \neg q(y)\wedge x+y=0^{\ast})$.

15- $\forall x\forall y\forall z( \neg q(x)\wedge \neg q(y)\wedge \neg q(z))\Rightarrow (x\times y)\times z+(y\times z)\times x+(z\times x)\times y=0^{\ast}$.

16- $\forall x\forall y\forall z(\neg q(y)\wedge \neg q(z))\Rightarrow x\times(y+z)=x\times y+x\times z$.

17- $\forall x\forall y(q(x)\wedge \neg q(y))\Rightarrow \neg q(x\times y)$.

18- $\forall x\forall y\forall z( q(x)\wedge \neg q(y)\wedge \neg q(z))\Rightarrow x\times (y\times z)=(x\times y)\times z= y\times (x\times z)$.

19- $\forall x\forall y\forall z( q(x)\wedge q(y)\wedge \neg q(z))\Rightarrow (x\times y)\times z= x\times (y\times z)$.

20- $\forall x\forall y\forall z( q(x)\wedge q(y)\wedge \neg q(z))\Rightarrow (x+ y)\times z= x\times z+ y\times z$.

21- for all $a\in L$ the axiom: $\neg q(a^{\ast})$.

22- for all $a, b\in L$ the axiom: $a^{\ast}\neq b^{\ast}$.

23- for all $a, b\in L$ the axioms: $(a+b)^{\ast}=a^{\ast}+b^{\ast}$ and $[a^{\ast}, b^{\ast}]=a^{\ast}\times b^{\ast}$.

24- for all natural number $m$ the axiom:  $\forall x\ (q(x)\wedge mx=0)\Rightarrow x=0 $.

25- $\forall x\ \neg q(x)\Rightarrow (\exists x_1\exists x_2 \ldots \exists x_n\ (\bigwedge_{i=1}^n q(x_i))\wedge (x=\sum_{i=1}^nx_i\times v_i))$.\\

Now, let $M\in Mod(\Sigma)$ be a model of $\Sigma$. Define
$$
K=\{ x\in M: \neg q(x)\}
$$
and
$$
E=\{x\in M: q(x)\}.
$$
Then it is easy to verify that $(K, E)$ is a completion of $L$. Conversely if $(K, E)$ is a completion of $L$ then we let $M=K\cup E$. For all $x, y\in M$, define $x+y$ to be their sum in $E$, if $x, y\in E$, their sum in $K$ if, $x, y\in K$, and $0^{\ast}$ otherwise. Similarly, define $x\times y$ to be $xy$, if $x, y\in E$, $[x,y]$, if $x, y\in K$, and the scalar product otherwise. It can be verified that $M$ is a model for $\Sigma$ and so $\Sigma$ axiomatizes the class $\mathfrak{X}_L$.
\end{proof}

It is clear that the class of all Lie algebras over a fixed field $E$ is elementary. We show that the class of all fileds for a fixed rational Lie algebra is not elementary. Let $L$ be such a Lie algebra. By $\mathfrak{F}_L$ we denote the class of all fields $E$ such that $(L, E)$ is a completion of $L$. Note that every element of $\mathfrak{F}_L$ is a sub-ring of the centroid $\Gamma(L)$.

\begin{proposition}
Let $L$ be a finite dimensional Lie algebra over $\mathbb{Q}$. Then $\mathfrak{F}_L$ is not elementary.
\end{proposition}

\begin{proof}
Let $\mathcal{L}$ be a first order language and $\Sigma$ be a set of axioms of $\mathfrak{F}_L$. Since $L$ is countable, we may assume that $\mathcal{L}$ is also countable. Now $\mathbb{Q}$ is a model for $\Sigma$ and $|\mathbb{Q}|\geq |\mathcal{L}|$. So by the Lowenheim-Skolm theorem, for any infinite cardinal $\alpha$, the set $\Sigma$ has a model $E$ of size $\alpha$. Now
$$
\dim_{\mathbb{Q}}L=[E:\mathbb{Q}]\dim_E L,
$$
and the left hand side is finite. So $[E:\mathbb{Q}]$ is finite, contradicting $|E|=\alpha$.
\end{proof}

Similarly, one can prove that the class $\mathfrak{F}_L$ is not elementary for any abelian Lie algebra $L$.

\section{Entangled ideals}
In this section, we introduce the concept of {\em entangled} ideals in a tensor product $E\otimes_{\mathbb{Q}}L$ and we show that they have close connections with the completions of $L$. For a rational Lie algebra $L$ and any field $E$ of characteristic zero, we say that a Lie algebra $K$ is an $E$-completion of $L$, if $(K,E)$ is a completion for $L$.

\begin{definition}
An ideal $N$ of the $E$-algebra $E\otimes_{\mathbb{Q}}L$ is called entangled, if $N\cap (1\otimes L)=0$.
\end{definition}

\begin{proposition}
Let $K$ be an $E$-completion of $L$. Then there exists an entangled ideal $N\unlhd E\otimes_{\mathbb{Q}}L$ such that
$$
K=\frac{E\otimes_{\mathbb{Q}}L}{N}.
$$
Conversely, for any entangled ideal $N$, the Lie algebra $K$ defined by the above equality is an $E$-completion.
\end{proposition}

\begin{proof}
Define a map $\varphi:E\times L\to K$ by $\varphi(x, a)=xa$. This map is  $\mathbb{Q}$-bilinear and so there is a $\mathbb{Q}$-linear map $\varphi^{\ast}:E\otimes_{\mathbb{Q}}L\to K$ with $\varphi^{\ast}(x\otimes a)=xa$. Since $K=Span_E(L)$, so $\varphi^{\ast}$ is surjective. Also for any $\lambda\in E$, we have
$$
\varphi^{\ast}(\lambda(x\otimes a))=\varphi^{\ast}((\lambda x)\otimes a)= \lambda \varphi^{\ast}(x\otimes a),
$$
hence $\varphi^{\ast}$ is $E$-linear. Further
\begin{eqnarray*}
\varphi^{\ast}([x\otimes a, y\otimes b])&=& \varphi^{\ast}((xy)\otimes [a,b])\\
                                        &=&xy[a,b]\\
                                        &=&[\varphi^{\ast}(x\otimes a), \varphi^{\ast}(y\otimes b)].
\end{eqnarray*}
So, $\varphi^{\ast}$ is an epimorphism of $E$-Lie algebras. Let $N=\ker \varphi^{\ast}$. Then
$$
K\cong \frac{E\otimes_{\mathbb{Q}}L}{N}.
$$
Note that
$$
N=\{ \sum_i x_i\otimes a_i\in E\otimes_{\mathbb{Q}}L:\ \sum_ix_ia_i=0\},
$$
so if $1\otimes a\in N$, then $a=0$ and therefore $N$ is an entangled ideal.

Conversely, let $N\unlhd E\otimes_{\mathbb{Q}}L$ be an entangled ideal. Let $K$ be defined as above. Then the map $a\mapsto 1\otimes a+N$ is a $\mathbb{Q}$-embedding of $L$ in $K$. So, we may assume that $L\leq_\mathbb{Q}K$. Now clearly $K=Span_E(L)$ and so $K$ is an $E$-completion of $L$.
\end{proof}

\begin{example}
Let $L$ be a Lie algebra over $E$. Then in the same time $L$ is a Lie algebra over $\mathbb{Q}$. Let $I\unlhd_{\mathbb{Q}}L$, $f\in Aut_{\mathbb{Q}}(L)$, and $\sigma\in Aut(E)$. Define
$$
N(I,f,\sigma)=\{ \sum_i x_i\otimes a_i\in E\otimes_{\mathbb{Q}}I:\ \sum_i\sigma(x_i)f(a_i)=0\}.
$$
Note that the map $\varphi:E\times I\to L$ defined by $\varphi(x,a)=\sigma(x)f(a)$ is $\mathbb{Q}$-bilinear and so $N(I,f,\sigma)$ is well-defined. It can be easily seen that $N(I,f,\sigma)$ is an entangled ideal. So we have a completion
$$
K(I,f,\sigma)=\frac{E\otimes_{\mathbb{Q}}L}{N(I,f,\sigma)}.
$$
\end{example}

\begin{proposition}
Let $E$ be a finite extension of $\mathbb{Q}$ and $L$ be a finite dimensional Lie algebra over $E$. With the above notations, we have
$$
\dim_E K(I,f,\sigma)=\dim_E K(f(I), 1,1),
$$
and in the special case,
$$
dim_E K(L, f, \sigma)=\dim_E L.
$$
\end{proposition}

\begin{proof}
Define a $\mathbb{Q}$-linear map $\sigma\otimes f:E\otimes_{\mathbb{Q}}L\to E\otimes_{\mathbb{Q}}L$ by $(\sigma\otimes f)(x\otimes a)=\sigma(x)f(a)$. This is an invertible map and we have
$$
(\sigma\otimes f)(N(I,f,\sigma))=N(f(I),1,1).
$$
Therefore $\dim_{\mathbb{Q}}N(I,f,\sigma)=\dim_{\mathbb{Q}}N(f(I), 1,1)$. Now, since $[E:\mathbb{Q}]$ is finite, so
$$
\dim_EK(I,f,\sigma)=\dim_EK(f(I),1,1).
$$
As an special case, when $I=L$, we have $\dim_E K(L,f,\sigma)=\dim_E K(L,1,1)$. But by an easy argument, one can see that
$$
\dim_E N(L, 1,1)=([E:\mathbb{Q}]-1)\dim_E L,
$$
so we have
$$
\dim_E K(L,f,\sigma)=\dim_E L.
$$
\end{proof}

We can say more about the structure of $K(L,f,\sigma)$. Let $a\to \bar{a}$ be the embedding of $L$ in $K(L,f,\sigma)$. We have,

\begin{proposition}
By the above notations,
$$
K(L,f,\sigma)=\{ \bar{a}:\ a\in L\}
$$
and for all $x\in E$, $x.\bar{a}=\overline{f^{-1}(\sigma(x)f(a))}$.
\end{proposition}

\begin{proof}
Let $N=N(L,f,\sigma)$ and $X=\sum x_i\otimes a_i+N\in K(L,f,\sigma)$. We know that $\sum\sigma(x_i)f(a_i)\in L$, so for some $a\in L$, we have
$f(a)=\sum \sigma(x_i)f(a_i)$. Therefore $1\otimes a-\sum x_i\otimes a_i\in N$ and hence $X=\bar{a}$. Now, if $x\in E$, then
$$
x.\bar{a}=x\otimes a+N=1\otimes b+N,
$$
for some $b$. Clearly, $f(b)=\sigma(x)f(a)$, so $b=f^{-1}(\sigma(x)f(a))$. Hence
$$
x.\bar{a}=\overline{f^{-1}(\sigma(x)f(a))}.
$$
\end{proof}

Now, we show that nilpotency (solvablity) of $L$ implies nilpotency (solvablity) of its completions.

\begin{proposition}
Let $L$ be a rational Lie algebra and $K$ be a completion for $L$ over some field $E$. If $L$ is nilpotent, then $K$ is nilpotent of the same class. If $L$ is solvable, then $K$ is also solvable of the same derived length.
\end{proposition}

\begin{proof}
We know that there is an entangled ideal $N$ such that
$$
K=\frac{E\otimes_{\mathbb{Q}}L}{N}.
$$
So, for any $n$ we have
$$
\gamma_n(K)=\frac{E\otimes_{\mathbb{Q}}\gamma_n(L)+N}{N}, \ \ K^{(n)}=\frac{E\otimes_{\mathbb{Q}}L^{(n)}+N}{N}.
$$
This shows that, if $\gamma_n(L)=0$ ($L^{(n)}=0$), then $\gamma_n(K)=0$ ($K^{(n)}=0$). On the other hand, if $\gamma_d(K)=0$ ($K^{(d)}=0$), then since $L\subseteq K$, we must have $\gamma_d(L)=0$ ($L^{(d)}=0$).
\end{proof}

A Lie algebra $L$ over $\mathbb{Q}$ is potentially finite dimensional, if it has a finite dimensional completion. Using a theorem of Ado (see \cite{Jac}), it is easy to prove that $L$ is potentially finite dimensional, if and only if $L\leq_{\mathbb{Q}}\mathfrak{g}\mathfrak{l}_n(E)$ for some field $E$ of characteristic zero and a natural number $n$. By the above proposition, we have

\begin{corollary}
Let $L$ be a potentially finite dimensional nilpotent rational Lie algebra. Then $L\leq_{\mathbb{Q}}\mathfrak{n}^+_n(E)$, for some field $E$ and $n\geq 1$, where $\mathfrak{n}^+_n(E)$ is the Lie algebra of strictly upper triangular $n\times n$ matrices over $E$.
\end{corollary}

\end{document}